\documentclass[12pt]{article}

\usepackage{amsmath,amssymb,amsfonts}
\usepackage{amsthm}%
\usepackage{enumitem}
\usepackage{graphicx}
\usepackage[T1]{fontenc}
\usepackage{mathscinet}
\usepackage{mathrsfs}
\usepackage{enumerate}

\allowdisplaybreaks%

{\theoremstyle{plain}%
 \newtheorem{theorem}{Theorem}
 
 \newtheorem{lemma}{Lemma}%
}
{\theoremstyle{remark}
\newtheorem{remark}{Remark}
}
{\theoremstyle{definition}

}

\begin{document}

\begin{center}
 {\large Group maps on initial intervals}

 \ 

{\sc John M. Campbell}

\vspace{0.1in}

{\footnotesize Department of Mathematics and Statistics}

{\footnotesize Dalhousie University}

{\footnotesize Halifax, NS B3H 4R2}

{\footnotesize Canada}

\vspace{0.1in}

{\footnotesize {\tt jh241966@dal.ca}}

\vspace{0.1in}

\end{center}

\begin{abstract}
 Let $S$ denote a $k$-subset of $\mathbb{Z}$. Let $G$ be a (multiplicative) group of size $k$. Assume that there exists a function 
 $\varphi\colon S \to G$ that is injective and that satisfies the following: If $x, y, z \in S$ and $z = xy$, then $\varphi(z) = \varphi(x) 
 \varphi(y)$. In 1990, Forcade and Pollington conjectured that if $S = \{ 1, 2, \ldots, k \}$, then $G$ is abelian. This has been a 
 long-standing conjecture highlighted by Richard K.\ Guy as an unsolved problem, and this problem appears to have remained open, 
 with its open status having been noted by Caicedo et al.\ in 2021 [\emph{Electron.\ J.\ Combin.}]. We succeed in proving the Forcade--Pollington conjecture in the affirmative and in full generality. 
\end{abstract}

\vspace{0.1in}


\vspace{0.1in}


\section{Introduction}
 There are many important and deep problems concerning combinatorial and number-theoretic aspects
 about finite groups. In this direction, the theory of logarithmic splittings of finite abelian groups
 often gives rise to rich areas of study, and similarly with respect to 
 discrete logarithms associated with power-residue characters. 
 These topics provide a foundation that led 
 Forcade and Pollington \cite{ForcadePollington1990}, in 1990, to provide a conjecture
 on what they refer to as \emph{group maps}, as defined below. 

 A set $S$ consisting of $k$ integers is said to be \emph{logarithmic} \cite{Guy1989}
 if there exists a group $G$ of size $k$ and an injective function
\begin{equation}\label{varphicolon} 
 \varphi\colon S \to G
\end{equation}
 such that: If $x, y, z \in S$ and $z = xy$, 
 then $\varphi(z) = \varphi(x) \varphi(y)$ (letting the binary operation on $G$ be denoted multiplicatively). 
 Forcade and Pollington \cite{ForcadePollington1990} conjectured that: 
 In the special case $S = \{ 1, 2, \ldots, k \}$, the group $G$ given as the codomain
 in \eqref{varphicolon} is necessarily abelian (cf.\ \cite{Chandler1988,ForcadeLamoreauxPollington1986} and Remark \ref{remarkChandler} below). 
 This conjecture was highlighted as an unsolved problem by Richard K.\ Guy
 \cite{Guy1989}, 
 and, based on extant research works citing Forcade and Pollington's work \cite{BlackburnMcKee2012,BosekDebskiGrytczukSokolSleszynskaNowakZelazny2018,CaicedoChartierPach2021,DebskiGrytczuk2021}, 
 it seems that Forcade and Pollington's 
 conjecture has remained open. 
 Moreover, its status as an open problem was highlighted as late as 2021 \cite{CaicedoChartierPach2021}. 
 In this paper, we succeed in proving Forcade and Pollington's conjecture 
 in the affirmative, based on our extensive interactions with GPT-5.6 Pro. 

 Prior to this paper, some partial progress on the 
 Forcade--Pollington conjecture has been made. As noted by 
 Guy \cite{Guy1989}, Chandler and Hickerson
 independently proved that if a group $G$ of the above form were noncommutative, 
 then its center would be of index $4$, $6$, or $8$, and Hickerson also showed that such a group 
 would be such that $|G| > 9696$. 
 Moreover, Chandler \cite{Chandler1988} proved that if a group of the above form is of odd size, 
 then it is necessarily abelian. 
 Such partial results motivate our full solution to the Forcade--Pollington conjecture, as
 in Section \ref{secfull} below. 

\begin{remark}\label{remarkChandler}
 In 1986, Forcade, Lamoreaux, and Pollington \cite{ForcadeLamoreauxPollington1986}
 conjectured that: For each positive integer $k$, there {exists} an abelian group $(G, \ast)$
 of order $k$ and a {surjective} map $f\colon \{ 1, 2, \ldots, k \} \to G$
 such that if $ij \leq k$, then $f(ij) = f(i) \ast f(j)$. 
 This is {not} the same as the conjecture attributed above to 
 Forcade and Pollington in 1990 \cite{ForcadePollington1990} (cf.\ \cite{Chandler1988,ForcadeLamoreauxPollington1986}). 
\end{remark}
 
\section{A full solution}\label{secfull}
 To avoid confusion in regard to the use of abelian/non-abelian operations, we let groups be denoted multiplicatively (since groups are 
 not assumed to be abelian, for our purposes). 
 We also let $e$ denote the (multiplicative) identity element of a given group. 

 We henceforth let $k$ be a positive integer and let $G$ be a group of size $k$, and we assume that 
 there is an injective mapping 
\begin{equation}\label{assumedmap} 
 \varphi\colon \{ 1, 2, \ldots, k \} \to G
\end{equation}
 such that 
\begin{equation}\label{preserves}
 \varphi(ab) = \varphi(a) \varphi(b)
\end{equation}
 whenever $a$, $b$, and $ab$ 
 are in $\{ 1, 2, \ldots, k \}$. 
 A mapping of this form is referred to as a \emph{group map}, 
 according to 
 Forcade and Pollington \cite{ForcadePollington1990}. 

 We begin with some basic but required consequences of the existence of 
 the group map indicated in \eqref{assumedmap}. 

\begin{lemma}\label{lemlocal}
 The equality $\varphi(1) = e$ holds, and if 
 $a$ and $b$ are in $\{ 1, 2, \ldots, k \}$ and satisfy $ab \leq k$, 
 then $\varphi(a) \varphi(b) = \varphi(b) \varphi(a) = \varphi(ab)$. 
\end{lemma}

\begin{proof}
 From the relation in \eqref{preserves}, we find that $\varphi(1) = \varphi(1) \varphi(1)$, 
 and left-multiplication by $\big( \varphi(1) \big)^{-1}$ gives us that $\varphi(1) = e$. 

 Now, assume that $a b \leq k$. 
 The relation in \eqref{preserves} gives us that 
 $\varphi(a) \varphi(b) = \varphi(ab)$ and that $\varphi(b) \varphi(a) = \varphi(ba)$. 
 The members $a$ and $b$ in the domain in \eqref{assumedmap} 
 commute, with $\varphi(ab) = \varphi(ba)$, 
 and hence the equality of $\varphi(a) \varphi(b) = \varphi(b) \varphi(a)$. 
 \end{proof}

 Observe that $\varphi$ is necessarily bijective (since $\varphi$ is assumed to be injective and since
 $\{ 1, 2, \ldots, k \}$ and $G$ are of equal size), 
 but Lemma \ref{lemlocal} does not claim that \emph{each} pair of elements
 of $G$ commutes: Lemma \ref{lemlocal} proves that $\varphi(a)$ and $\varphi(b)$
 commute only with the additional assumption that $ab \leq k$. If $ab > k$, 
 then the definition of a group map does not, in any direct way, provide any relation
 between $\varphi(a) \varphi(b)$ and $\varphi(b) \varphi(a)$.

 We henceforward assume that $k \geq 2$, since we may disregard the trivial case given by groups of size $1$ (which 
 are necessarily abelian). 

 We write $t$ in place of the unique positive integer such that 
\begin{equation}\label{uniquepositive}
 t(t-1) \leq k < t(t+1)
\end{equation}
 with 
\begin{equation}\label{floorsqrt}
 t = \left\lfloor \frac{1 + \sqrt{1 + 4 k }}{2} \right\rfloor. 
\end{equation}

\begin{lemma}\label{lemsubabelian}
 The subgroup $\langle \varphi(1), \varphi(2), \ldots, \varphi(t) \rangle $ of $G$ is abelian. 
\end{lemma}

\begin{proof} 
 Let $i$ and $j$ be members of $\{ 1, 2, \ldots, t \}$. 
 For the case whereby $i = j$, we find that $\varphi(i)$ 
 and $\varphi(j)$ commute, i.e., since a group element commutes with itself. 
 If $i \neq j$, then the largest possible product of two unequal elements
 of $\{ 1, 2, \ldots, t \}$ is $t(t-1)$, 
 so that $i j \leq t(t-1) \leq k$. 
 From Lemma \ref{lemlocal}, 
 we find that $\varphi(i) \varphi(j) = \varphi(j) \varphi(i)$. 
 A group generated by pairwise commuting elements is abelian, and hence the desired result. 
\end{proof}

 Our strategy, at this point, is to show that 
 $\langle \varphi(1), \varphi(2), \ldots, \varphi(t) \rangle $ 
 cannot be a proper subgroup of $G$. 

 Define
 $ \mathcal{R}_{t}
 = \{ (a, b) \in \mathbb{Z}^{2} : 1 \leq a \leq t, 1 \leq b \leq t-1, 
 \gcd(a, b) = 1 \}$, 
 and define
\begin{equation}\label{Thetatcolon}
 \Theta_t\colon \mathcal{R}_{t} \to \langle \varphi(1), \varphi(2), \ldots, \varphi(t) \rangle
\end{equation}
 so that $\Theta_t(a, b) = \varphi(a) \big( \varphi(b) \big)^{-1}$
 for a pair $(a, b)$ in $\mathcal{R}_{t}$. 
 Observe that since $a, b \leq t$ for a pair $(a, b) \in \mathcal{R}_{t}$, each expression 
 of the form $\Theta_t(a, b)$ is indeed in the codomain 
 indicated in \eqref{Thetatcolon}. 

\begin{lemma}\label{lemThetat}
 The mapping $\Theta_t$ is injective. 
\end{lemma}

\begin{proof}
 Assume that $\Theta_t(a, b) = \Theta_t(c, d)$
 for pairs $(a, b)$ and $(c, d)$ in the domain shown in 
 \eqref{Thetatcolon}. 
 We thus find that 
\begin{equation}\label{frominjective}
 \varphi(a) \big( \varphi(b) \big)^{-1} = \varphi(c) \big( \varphi(d) \big)^{-1}. 
\end{equation} 
 Each element among $\varphi(a)$, $\varphi(b)$, $\varphi(c)$, and $\varphi(d)$
 is in the abelian group $\langle \varphi(1), \varphi(2), \ldots, \varphi(t) \rangle$. 
 So, from 
 \eqref{frominjective}, we write 
\begin{equation}\label{rearrangeinjective} 
 \varphi(a) \varphi(d) = \varphi(c) \varphi(b). 
\end{equation} 
 Now, since $a, c \leq t$ and $b, d \leq t - 1$, 
 we find that 
 $ad \leq t(t-1) \leq k$ and that $cb \leq t(t-1) \leq k$. 
 From the assumption that $\varphi$ is a group map, 
 we find that $\varphi(a) \varphi(d) = \varphi(ad)$ and that $\varphi(c) \varphi(b) = \varphi(cb)$. 
 From \eqref{rearrangeinjective}, we thus find that 
 $\varphi(ad) = \varphi(cb)$, and the assumption that $\varphi$ is injective
 then gives us that 
\begin{equation}\label{adequalscb}
 ad = cb.
\end{equation} 
 Moreover, since $(a, b)$ and $(c, d)$ are in $\mathcal{R}_{t}$, we find that 
 $\gcd(a, b) = \gcd(c, d) = 1$. 
 The equality in \eqref{adequalscb} then gives us that $a \mid c b$. 
 Moreover, since $a$ and $b$ are relatively prime, 
 we can conclude that $a \mid c$. A symmetric argument
 then gives us that $c \mid a$. Since $a$ and $c$ are positive integers, 
 we thus find that $a = c$. Consequently, the relation in \eqref{adequalscb}
 allows us to deduce that $b = d$, 
 i.e., so that $(a, b) = (c, d)$, as desired. 
\end{proof}

\begin{lemma}\label{lemtriangular}
 For each integer $t \geq 5$, the relation 
 $ \left| \mathcal{R}_{t} \right| \geq \frac{t(t+1)}{2} $ 
 holds. 
\end{lemma}

\begin{proof}
 Define 
\begin{equation}\label{defineDeltat}
 \Delta_{t} := \{ (x, y) \in \mathbb{Z}^{2} : 1 \leq x \leq y \leq t \}. 
\end{equation}
 We then find that $|\Delta_{t}| = \frac{t(t+1)}{2}$. 
 Our strategy, at this point, is to construct an injective map
 from $\Delta_{t} \to \mathcal{R}_{t}$. 

 For a pair $(x, y)$ in $\Delta_t$, write $\delta = \gcd(x, y)$ and 
\begin{equation}\label{xandy}
 x = \delta r \quad \text{and} \quad y = \delta s,
\end{equation} 
 where 
\begin{equation}\label{gcdrs1}
 \gcd(r, s) = 1
\end{equation}
and $1 \leq r \leq s$. 
 Now, define
\begin{equation}\label{originalFt}
 F_t(x, y) = \big( r + (\delta - 1) s, s \big) 
\end{equation}
 and write 
\begin{equation}\label{rewriteFtxy}
 F_t(x, y) = (u, v). 
\end{equation} 
 We thus have that $v = s$ and that $1 \leq v = s \leq y \leq t$. 
 Since $r \leq s$, we have that 
\begin{equation}\label{fromrleqs}
 u = r + (\delta - 1) s \leq s +(\delta - 1) s = \delta s = y \leq t. 
\end{equation} 
 Observe that $ \gcd(u, v) = \gcd(r + (\delta - 1) s, s ) = \gcd(r, s) = 1$. 

Since
\begin{equation}\label{displayEuclidean}
u-1=(\delta-1)v+(r-1)
\end{equation}
and
$
0\leq r-1<v,
$
 the relation in \eqref{displayEuclidean} gives 
 the Euclidean division of $u-1$ by $v$. Hence the
quotient and remainder are $\delta-1$ and $r-1$, respectively. 
In particular, we have that 
\begin{equation}\label{quotientremainder}
\delta =
1+\left\lfloor\frac{u-1}{v}\right\rfloor
\qquad\text{and}\qquad
r=u-(\delta-1)v.
\end{equation}

Moreover, $s=v$. Thus the pair $(u,v)$ uniquely determines
$
x=\delta r$ and 
 $y=\delta s=\delta v$. As a consequence, we have that $F_t$ is injective.

 Now, for an ordered pair $(x, y)$ in $\Delta_t$, define 
\begin{equation}\label{definePsit}
 \Psi_{t}(x, y) = \begin{cases} 
 (t-1, t-2), & \text{if $(x, y) = (t, t)$,} \\
 (t, u), & \text{if $v=t$,} \\
 (u, v), & \text{otherwise,} 
 \end{cases}
\end{equation}
 being consistent with the notation in \eqref{rewriteFtxy}. 

 We claim that each expression of the form $ \Psi_{t}(x, y)$ is in $\mathcal{R}_{t}$. 
 To begin with, if $(x, y) = (t, t)$, 
 then, since $1 \leq t - 2 \leq t- 1$ and $t - 1 \leq t$ and $\gcd(t-1, t-2) = 1$, 
 we find that $(t-1,t-2) \in \mathcal{R}_{t}$. 
 Now, if $v = t$, then, since the latter entry of $F_{t}(x, y)$ is $s$
 (recalling \eqref{originalFt}), we find that $s = t$, 
 but 
 $y = \delta s = \delta t \leq t$, and hence the equality $\delta = 1$, giving us that $y = t$. 
 So, we find that $u = r = x$ and that $\gcd(u, t) = 1$. 
 Since $1 \leq u \leq t$ and $\gcd(u, t) = 1$, it cannot be the case that $u$ equals $t$, 
 because, otherwise, we would have that 
 $\gcd(u, t) = \gcd(t, t) = t > 1$. We thus find that $1 \leq u \leq t - 1$, 
 so that $(t, u) \in \mathcal{R}_{t}$.

 Finally, suppose that $(x, y)$ falls under the third case in the definition of $\Psi_{t}$. 
 Then $v$ and $t$ are distinct. Moreover, since $1 \leq v \leq t$, we see that $1 \leq v \leq t-1$. 
 Consequently, the pair $(u, v)$ is in $\mathcal{R}_{t}$. 
 This allows us to restrict the codomain of $\Psi_t$, writing 
\begin{equation}\label{Psitcolon}
 \Psi_t\colon \Delta_t \to \mathcal{R}_{t}. 
\end{equation}

 We claim that the mapping in \eqref{Psitcolon} is injective, and this leads us to the following case analysis
 (where the cases referenced below refer to the ordering of the cases given by the definition of the piecewise
 function in \eqref{definePsit}). 

 \ 

\noindent {\bf Case 1:} The equality of expressions of the form $(t-1, t-2)$ 
 based on possible arguments of the form $(x, y) = (t, t)$ implies the equality of the corresponding arguments.

 \ 

\noindent {\bf Case 2:} Recalling the notational convention in \eqref{rewriteFtxy}, 
 transposed expressions of the form $(v, u)$ are mutually distinct for distinct arguments in the domain 
 in \eqref{Psitcolon}, by the injectivity of $F_t$. 

 \ 

\noindent {\bf Case 3:} Expressions of the form $\Psi_t(x, y) = F_t(x, y)$
 are mutually distinct for distinct arguments in the domain in \eqref{Psitcolon}, 
 again by the injectivity of $F_t$. 

 \ 

\noindent {\bf Cases 1 \& 2:} An expression of the form $(t-1, t-2)$ corresponding to the first case in 
 \eqref{definePsit} cannot be equal to an expression of the form $(t, u)$ corresponding to 
 the second case in \eqref{definePsit}. 

 \ 

\noindent {\bf Cases 2 \& 3:} We claim that no pair of the form $(t, u)$ given by the second case in 
 \eqref{definePsit} can equal a pair of the form $(u, v)$
 given by the final case in \eqref{definePsit}. 
 To prove this claim, suppose that $F_t(x, y)$ is such that its first coordinate is $t$.
 Thus, for some integer $w$, we have that 
\begin{equation}\label{assumetv} 
 F_{t}(x, y) = (t, w).
\end{equation} 
 We remain consistent with the definitions associated with $\delta$, $r$, and $s$
 given above. 
 By the definition of $F_t$, we have $F_t(x, y) = (r + (\delta - 1)s, s)$. 
 Comparing the entries of this tuple
 with \eqref{assumetv}, 
\begin{equation}\label{fromFteqs} 
 t = r + (\delta - 1) s. 
\end{equation}
 From \eqref{fromrleqs} and \eqref{fromFteqs}, 
 we have that 
\begin{equation}\label{yequalst}
 y = t
\end{equation}
 and $r + (\delta - 1) s = \delta s$, and this latter equality gives us, in turn, that $r = s$. 
 Moreover, the coprimality condition in 
 \eqref{gcdrs1} then gives us, as a consequence, that 
\begin{equation}\label{rs1} 
 r = s = 1, 
\end{equation}
 so that \eqref{assumetv} reduces (recalling \eqref{yequalst}) to 
\begin{equation}\label{changetv} 
 F_{t}(x, t) = (t, 1).
\end{equation} 
 In turn, a combined application of \eqref{xandy}, \eqref{yequalst}, and 
 \eqref{rs1} allows us to deduce that $x = t$, i.e., so that \eqref{changetv} reduces to 
\begin{equation}\label{Fttt}
 F_{t}(t, t) = (t, 1).
\end{equation} 
 However, the argument $(t, t)$ is given by the first case in the definition of $\Psi_t$, with $\Psi_{t}(t, t) = (t-1, t-2)$. So, the argument 
 $(t, t)$ does not yield a pair given by the third case in the definition of $\Psi_t$. We have shown that $(t, t)$ is the only argument such 
 that when $F_t$ is applied to it, this results in a pair with $t$ as the first entry. On the other hand, no pair arising from the third case in the 
 definition of $\Psi_t$ has $t$ as the first coordinate. Each pair arising from the second case has $t$ as the first entry, and we thus find 
 that no pair arising from the second case can equal a pair arising from the third case. 

 \ 

\noindent {\bf Cases 1 \& 3:} It remains to verify that no tuple of the form $(u, v)$ given by the final case of \eqref{definePsit} is of the 
 form $(t-1,t-2)$. By way of contradiction, suppose that 
\begin{equation}\label{equalsfirstcase} 
 F_{t}(x, y) = (t - 1, t - 2). 
\end{equation}
 The formula for $\delta$ in \eqref{quotientremainder}, together with the assumption in \eqref{equalsfirstcase}, gives us that $\delta 
 = 2$. Moreover, the formula for the remainder $r$ in \eqref{quotientremainder} and the assumption in \eqref{equalsfirstcase} give us that 
 $ r = 1$. Recalling \eqref{originalFt}, we also have that $s = t - 2$. From \eqref{xandy}, we also have that $y = 2 t - 4$. With the 
 assumption that $t \geq 5$, we find that $2t-4 > t$, i.e., that $y > t$, contradicting the final inequality in \eqref{defineDeltat}, i.e., 
 contradicting that $(x, y) \in \Delta_t$. 

 \ 

\noindent So, we have shown that $\Psi_t$ is injective. By then considering the cardinalities of the codomain and the domain in 
 \eqref{Psitcolon}, we obtain the desired result. 
\end{proof}

\begin{theorem}
 If $k \geq 20$, then Forcade and Pollington's conjecture holds. 
\end{theorem}

\begin{proof}
 Assume that $k \geq 20$. The relation between $k$ and $t$, as in \eqref{floorsqrt}, then gives us that $t \geq 5$. From Lemma 
 \ref{lemThetat}, we find that 
\begin{equation}\label{boundsubgroup} 
 \left| \mathcal{R}_{t} \right| \leq \left| \langle \varphi(1), \varphi(2), \ldots, \varphi(t) \rangle \right|. 
\end{equation}
 By Lemma \ref{lemtriangular} (recalling that $t \geq 5$) and the relation in 
 \eqref{boundsubgroup}, we find that 
\begin{equation}\label{triangularbound} 
 \frac{t(t+1)}{2} \leq \left| \langle \varphi(1), \varphi(2), \ldots, \varphi(t) \rangle \right|. 
\end{equation}
 From \eqref{uniquepositive} and \eqref{triangularbound} together, we find that 
 $\frac{k}{2} < \big| \langle \varphi(1), \varphi(2), \ldots, \varphi(t) \rangle \big|$, 
 i.e., that 
\begin{equation}\label{contradictLagrange}
 \frac{|G|}{2} < \big| \langle \varphi(1), \varphi(2), \ldots, \varphi(t) \rangle \big|.
\end{equation}
 If $ \langle \varphi(1), \varphi(2), \ldots, \varphi(t) \rangle $ were a proper subgroup of $G$, then its index would be such that $\big[G : 
 \big| \langle \varphi(1), \varphi(2), \ldots, \varphi(t) \rangle \big| \big] \geq 2$, so that Lagranges theorem would thewn give us that 
 $ \big| \langle \varphi(1), \varphi(2), \ldots, \varphi(t) \rangle \big| \leq \frac{|G|}{2}$, contradicting \eqref{contradictLagrange}. We thus 
 have that $ \langle \varphi(1), \varphi(2), \ldots, \varphi(t) \rangle = G$. Consequently, by Lemma \ref{lemsubabelian}, 
 we have that $G$ is abelian. 
\end{proof}

\begin{theorem}
 Forcade and Pollington's conjecture holds in full generality. 
\end{theorem}

\begin{proof}
Suppose, by way of contradiction, that there exists a
counterexample $G$ to the Forcade--Pollington conjecture.
As reported by Guy \cite{Guy1989}, Hickerson proved that every
such counterexample satisfies
$
|G|>9696$. 
In particular, $|G|\geq20$. The preceding theorem therefore
implies that $G$ is abelian, contradicting the assumption that
$G$ is a counterexample. Hence no counterexample exists, and
the conjecture holds in full generality.
\end{proof}

\subsection*{Acknowledgements}
 The author acknowledges extensive interactions with GPT-5.6 Pro during the exploratory and proof-development stages of this work. All 
 AI-generated suggestions were substantially revised, corrected, and independently verified by the author, who assumes full responsibility 
 for the mathematical content.

\bibliographystyle{plain}
\bibliography{seprefe}

\end{document}